\documentclass{amsart}
\usepackage{amsmath}
\usepackage{amssymb}
\usepackage{graphicx} 

\newtheorem{theorem}{Theorem}[section]

\newtheorem{lemma}{Lemma}[section]

\def\ad#1{\begin{aligned}#1\end{aligned}}  \def\b#1{{\mathbf{#1}}}
\def\a#1{\begin{align*}#1\end{align*}} \def\an#1{\begin{align}#1\end{align}} 
 \def\t#1{{\operatorname{#1}}}
\def\p#1{\begin{pmatrix}#1\end{pmatrix}}  \numberwithin{equation}{section}
 \def\boxit#1{\vbox{\hrule height1pt \hbox{\vrule height 10pt width1pt depth2pt
    \kern1pt \small #1\kern1pt\vrule width1pt}\hrule height1pt }}
 \def\lab#1{\ \boxit{ #1 }\ \label{#1}}   \def\d{\operatorname{div}}
      
\long\def\myskip#1{}  
   \def\lab#1{\label{#1}}

\begin{document} \baselineskip=16pt\parskip=4pt

\title[nonconforming finite element]
 {A nonconforming P2 and discontinuous P1 mixed finite element on tetrahedral grids }

 \author {Shangyou Zhang}
 
\address{Department of Mathematical Sciences, University of Delaware,
    Newark, DE 19716, USA. }
\email{szhang@udel.edu}




\begin{abstract}
 A nonconforming $P_2$ finite element  is constructed by enriching
   the conforming $P_2$ finite element space with seven
   $P_2$ nonconforming bubble functions  (out of fifteen such bubble functions on
     each tetrahedron).
This spacial nonconforming $P_2$ finite element, 
   combined with the discontinuous $P_1$ finite element
   on general tetrahedral grids, is inf-sup stable for solving the Stokes equations.
Consequently such a mixed finite element method produces optimal-order convergent 
  solutions for solving the stationary Stokes equations. 
  Numerical tests confirm the theory.
 
\end{abstract}

\maketitle

\section{Introduction}
We solve the Stokes equations:  Find functions
  $\b u $, the fluid velocity, and $p$ , the pressure in the flow, on a  
  3D polyhedral domain  $\Omega$ such that
  \an{\label{e1-1} &&  - \Delta \b u  + \nabla p
            &=\b f \qquad && \hbox{in } \Omega, && \\
       \label{e1-2}        &&  \d \b u  &= 0 && \hbox{in } \Omega, \\
       \label{e1-3}        &&  \b u  &= \b 0  && \hbox{on } \partial\Omega,  }
  where $\b f \in L^2(\Omega)$.  The variational form of \eqref{e1-1}--\eqref{e1-3} reads:  
    Find $\b u\in\b V=H_0^1(\Omega)^3$ and 
    $p\in P= L^2_0(\Omega)$  such that 
\a{ \ad{  (\nabla \b u, \nabla \b v)- (\d \b v,p) &=(\b f,\b v)  
             &&   \forall \b v \in \b V, \\
        (\d \b u,q)  &=0   &&  \forall q \in P. }
   }  
The problem is symmetric, but not positive definite.
When choosing the finite element spaces for $V$ and $P$, most pairs are not stable,
  i.e., they fail in satisfying the inf-sup condition \eqref{inf} below. 
A natural pair of finite elements is
  the  $P_k^{\t{c}}$/$P_{k-1}^{\t{dis}}$ mixed element where $P_k^{\t{c}}=\b V_h\subset \b V$
   is the space of continuous polynomials of a tetrahedral mesh, and $P_{k-1}^{\t{dis}}=P_h
    \subset P$ is the space of discontinuous $P_{k-1}$ polynomials on the mesh.
Mostly such a method is not stable.  But on Hsieh-Clough-Tocher macro 
  triangular/tetrahedral grids
   \cite{Guzman-Neilan,Qin,Xu-Zhang,Zhang-3D},
  the full $P_k^{\t{c}}$/$P_{k-1}^{\t{dis}}$ space, $k\ge 2$ in 2D or $k\ge 3$ in 3D, 
   is inf-sup stable.
In all other known stable cases, the pressure space is a proper subspace of 
  the $P_{k-1}^{\t{dis}}$ space  
\cite{Arnold-Qin,Bacuta,Fabien-Neilan,Falk-Neilan,Fu-Guzman-Neilan,Guzman-Lischke-Neilan,Huang-Q2,Neilan,Scott-Vogelius,Scott-V,Zhang-3D,Zhang-ps2,Zhang-Qk,Zhang-3d-P2,Zhang-6},
except when the discrete velocity is enriched by non-polynomial bubbles or subgrid-bubbles
   \cite{Guzman-Neilan1,Guzman-Neilan2}.

On the other side,  it is relatively easy to find stable 
   nonconforming $P_k^{\t{nc}}$/$P_{k-1}^{\t{dis}}$
   pairs of finite elements, especially in low polynomial degree cases.
Here $P_k^{\t{nc}}$ means the nonconforming finite element space of polynomial degree $k$,
   where the piecewise polynomials are continuous up to $P_{k-1}$ order in the sense that
   the jump of function is orthogonal to $P_{k-1}$ polynomials on inter-element face, 
    $\int_F [f]p_{k-1}ds=0$.
Crouzeix and Raviart proposed nonconforming finite elements first 
   \cite{Crouzeix-Raviart}.   
In this first paper,  $P_1^{\t{nc}}$/$P_{0}^{\t{dis}}$ elements are proved to be inf-sup stable, 
   on triangular and tetrahedral grids.  
The proposed 2D $P_3^{\t{nc}}$/$P_{2}^{\t{dis}}$ element \cite{Crouzeix-Raviart}
    is enriched by three $P_4^{\t{nc}}$-bubbles for each component of $\b v$.
It is proved that higher-order bubbles are not needed in \cite{Crouzeix-Falk} if the   
     triangular grid  can be separated in to macro-triangles of several patterns.

Fortin and Soulie studied 2D $P_2^{\t{nc}}$/$P_1^{\t{dis}}$ mixed finite elements \cite{Fortin-2D}
  and proved the inf-sup stability, without enriching the $P_2^{\t{nc}}$ element by any 
    higher-order bubbles.
For 2D $P_k^{\t{nc}}$/$P_{k-1}^{\t{dis}}$  elements,  Matthies and Tobiska enrich the
   velocity space by many higher-order nonconforming bubble functions so that the method
   is stable for all $k\ge 1$ \cite{Matthies-Tobiska}.
But \cite{Baran} add only one $P_k^{\t{nc}}$-bubble (not high-polynomial bubbles like
    \cite{Matthies-Tobiska}) to each
  component of $P_k$ conforming finite element velocity each triangle 
   so that the $P_k^{\t{nc}}$/$P_{k-1}^{\t{dis}}$ is stable for even polynomial 
      degree $k\ge 2$.
When $k=2$, \cite{Baran} repeats the result of Fortin and Soulie \cite{Fortin-2D}.

In 3D, there are only two works \cite{Ciarlet} and  \cite{Fortin}
   on the $P_k^{\t{nc}}$ element, other than the work of first element
   $P_1^{\t{nc}}$/$P_{0}^{\t{dis}}$ in \cite{Crouzeix-Raviart}.
Both \cite{Ciarlet} and  \cite{Fortin} study the 3D $P_2^{\t{nc}}$ element
  on tetrahedral grids.
But none constructs a working $P_2^{\t{nc}}$ finite element space
  for the Stokes equations, while this work is the
  first one accomplishing the work.
In \cite{Fortin}, Fortin finds five $P_2^{\t{nc}}$ bubbles for the scalar 
     $P_2^{\t{c}}$ finite element space on each tetrahedron.
But \cite{Fortin} only suggests to use  (some of) these bubbles to stabilize 
   the $P_2^{\t{c}}$ finite element, and does not propose a working method.
In \cite{Ciarlet}, Ciarlet, Dunkl and Sauter find (all) $P_k^{\t{nc}}$ bubbles 
   for the $P_k^{\t{c}}$ finite element space in 3D, $k \ge 1$.
When $k=2$,  the bubbles of \cite{Fortin} and \cite{Ciarlet} are same.
\cite{Ciarlet} further points out the $P_k^{\t{c}}$ space and
   the (found) $P_k^{\t{nc}}$-bubble space are not linearly independent.
\cite{Ciarlet} suggests to remove all the vertex basis functions of
 the $P_k^{\t{c}}$ finite element to obtain a uni-solvent $P_k^{\t{nc}}$ finite element space.
But this $P_2^{\t{nc}}$ finite element space is not inf-sup stable when combined
  with the $P_1^{\t{dis}}$ pressure space for the Stokes equations.

\def\ncb{\text{$P_2^{\t{nc}}$}} \def\cq{\text{$P_2^{\t{c}}$}} 

In this work,  we propose to add seven, out of fifteen,  
   vector $P_2^{\t{nc}}$-bubbles of \cite{Ciarlet,Fortin} to the
   $P_2^{\t{c}}$ space, on one tetrahedron.
It is shown such a 3D $P_2^{\t{nc}}$/$P_{1}^{\t{dis}}$
   finite element is inf-sup stable for the Stokes equations. 
The unique solution of such finite element Stokes equations converges at the optimal order.
This work closes a 50-year open problem that if the 3D $P_2^{\t{nc}}$/$P_{1}^{\t{dis}}$
  finite element is stable, appeared in the 1973 paper of  Crouzeix-Raviart 
  \cite{Crouzeix-Raviart}, the first paper 
   on the finite element methods for Stokes equations.

Recently the above mentioned unstable
   $P_2^{\t{nc}}$/$P_1^{\t{dis}}$ element (enriched by three $P_2^{\t{nc}}$-bubbles 
    each tetrahedron) is published in Math.~Comp.~\cite{Sauter} as a stable 
    finite element.
The paper \cite{Sauter} simply quoted the Stenberg macro-element theorem \cite{Stenberg} as a proof.
But \cite{Sauter} missed a condition of the Stenberg macro-element theorem 
   that there must be an internal degree of freedom inside each face-triangle. 
The $P_2^{\t{nc}}$/$P_1^{\t{dis}}$ element of \cite{Sauter} is not stable on general tetrahedral meshes,
  but it can be stable on some structured meshes.
For example, it is stable on a tetrahedral mesh refined from a flat-face hexahedral 
   mesh in a special way \cite{Zhang-Shuo}, as $P_2^{\t{c}}$/$P_0^{\t{dis}}$ element is proved stable on
     such tetrahedral meshes.

The rest of the paper is organized as follows.
In Section 2, we construct the \ncb-bubble functions on general tetrahedra.  
   We define the proposed $P_2^{\t{c}}$/$P_{1}^{\t{dis}}$
   finite element for the Stokes equations. 
In Section 3, we  prove the stability of the $P_2^{\t{c}}$/$P_{1}^{\t{dis}}$ finite element.
In Section 4, we show the finite element solutions converge at the optimal order. 
In Section 5, we provide some numerical results.

\section{The \ncb-bubbles and the $P_2^{\t{nc}}$/$P_{1}^{\t{dis}}$ element }

Let $\mathcal{T}_h=\{T_1,\cdots,T_{n_t}\}$ be a quasiuniform  
	tetrahedral grid on a 3D polyhedral $\Omega$.
Let $\mathcal{V}_h^0=\{\b x_1,\cdots,\b x_{n_v}\}$ 
   be the set of interior vertex
   in the grid $\mathcal{T}_h$.
Let $\mathcal{E}_h^0=\{\b e_1,\cdots,\b e_{n_e}\}$ 
    be the set of interior edges (mid-point inside $\Omega$) 
   in the grid $\mathcal{T}_h$.
Let $\mathcal{F}_h$ be the set of 
     triangles in the grid $\mathcal{T}_h$. 
Let $\mathcal{F}_h^0=\{F_1,\dots,F_{n_f}\}$ be the set of 
   interior triangles (whose bary-center are inside $\Omega$) 
   in the grid $\mathcal{T}_h$. 
 
 \begin{figure}[htb] 
  \setlength{\unitlength}{1pt}

 \begin{center}\begin{picture}(  250.,  130.)(  0.,  0.)
     \def\lb{\circle*{0.8}}\def\lc{\vrule width1.2pt height1.2pt}
     \def\la{\circle*{0.3}}
   \put(-12,1){$\b x_1$}\put(152,41){$\b x_2$}\put(55,118){$\b x_3$}
       \put(52,43){$\b x_4$} \put(10,90){$T:$} \put(60,60){$F_4$} 
     \multiput(   1.56,   1.25)(   2.343,   1.874){ 20}{\la}
     \multiput( 148.00,  40.00)(  -3.000,   0.000){ 32}{\la}
     \multiput(  50.00, 118.00)(   0.000,  -3.000){ 26}{\la}
     \multiput(   0.00,   0.00)(   0.096,   0.231){520}{\la}
     \multiput(   0.00,   0.00)(   0.242,   0.064){620}{\la}
     \multiput( 150.00,  40.00)(  -0.195,   0.156){512}{\la}

   \put(180,90){$\lambda_i\in P_1(T)$, }
   \put(180,75){$\lambda_i(\b x_j)=\delta_{ij}, $}
   \put(180,60){$\lambda_i|_{F_i}=0. $}

 \end{picture}\end{center}

\caption{A tetrahedron $T$, its vertices $\b x_i$ and its barycentric coordinates $\lambda_i$.}
\label{T}
\end{figure}
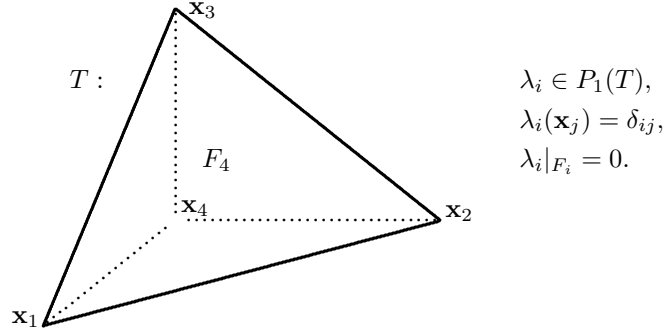 
 
Let $T\in\mathcal{T}_h$ have vertices $\{\b x_i, i=1,\dots,4\}$, cf. Figure \ref{T}. 
Let $\lambda_i$ be a barycentric coordinate of $T$, i.e., 
\a{  \lambda_i\in P_1(T), \  \lambda_i(\b x_j)=\delta_{ij}, \ i,j=1,\dots,4.  }
Let $F_i$ be a face-triangle of $T$, opposite to the vertex $\b x_i$.   
We have $\lambda_i|_{F_i}=0 $.
By \cite{Fortin}, the central \ncb-bubble on $T$ is 
\an{\label{Phi0} \Phi_0(\b x) = 2-4(\lambda_1^2+\lambda_2^2+\lambda_3^2+\lambda_4^2). }
$\Phi_0(\b x)$ is the unique $P_2$ polynomial satisfying
\an{\label{dof} \int_{F_i} \Phi_0 \lambda_j dS &=0, \  i,j=1,\dots,4,  \ j\ne i,\quad
      \Phi_0 (\b e_T)
        =1,  }
where  $\b e_T=(\b x_1+\b x_2+\b x_3+\b x_4)/4$.
By modifying Fortin's bubbles \cite{Fortin}, the face \ncb-bubbles on $T$ are defined as
\an{\label{Phi1} \Phi_i(\b x) =12(1-\lambda_i)^2-18
    (\lambda_1^2+\dots+\widehat{\lambda_i^2}+\dots+\lambda_4^2)-\left(\frac 32\right)^3 \Phi_0, }
$i=1,\dots,4$.  
Here $\widehat{\lambda_i^2}$ means the term is dropped from the sum.
$\Phi_i(\b x)$ is the unique $P_2$ polynomial satisfying
\a{ & \int_{F_j} \Phi_i \lambda_m dS 
            =\begin{cases} 0, &  m,j=1,\dots,4,  \ j,m\ne i,\\
                              |F_j|, & m=1,\dots,4,  \ m\ne i, j=i, \end{cases} 
    \quad  \Phi_i (\b e_T)
       =0,  }
$i=1,\dots,4$, where $\b e_T$ is the bary-center of $T$. 

The nodes of \ncb \ basis functions are bary-centers and mid-face-triangle points of all
   $T\in \mathcal{T}_h$.
Let $\b b_T$ and $\{\b x_i, \ i=1,\dots,4\}$ be the
bary-center and the mid-triangle points of $T\in\mathcal{T}_h$, respectively.
We adopt the following 7 $P_2^{\t{nc}}$ bubbles: 
  3 \ncb-bubbles $\Phi_0\b E_m$ (defined after \eqref{g-2} below) at the bary-center $\b b_{i_1}$
     of $T_{i_1}$ and 4 \ncb-bubbles $\Phi_j\b n_j$ associated with
     the 4 face-triangles $\{ F_j, \ j=1,\dots,4\}$ of $T_{i_2}$.
Here $T_{i_1}$ is the tetrahedron on which the $i$-th global basis function $\phi_i$  
  of the first type $P_2^{\t{nc}}$ bubble  is supported on, 
  and $T_{i_2}$ is any one of the two tetrahedra whose
      interface triangle $F_j$ is the supporting triangle of the globally $i$-th
       basis function $\phi_i$ of the second type $P_2^{\t{nc}}$ bubble. 
Together, globally we list them as, excluding those basis functions at domain boundary, 
\an{ \label{g-2}
   \ad{&& \phi_i &=\begin{cases} \Phi_0 \b E_m, & \t{on }\ T, \ \t{if } \ \b b_{i_1}
                               \in T, \\
                         0 , & \t{on }\ T, \ \t{if } \ \b b_{i_1} \not\in T, \end{cases}
     \ && i=1,\dots,n_1, \\
  &&\phi_i &=\begin{cases} \Phi_j \b n_j, & \t{on }\ T_{i_2}, \ \t{if } \ F_{j}
                               \in T_{i_2}, \ F_{j}\not\subset\partial\Omega,\\
                         0 , & \t{on }\ T, \ \t{if } \ F_j \not\subset T, \end{cases}
     && i=n_1+1,\dots,n_{nc},  
  } }  
where $\b b_{i_1}$ and $T_{i_1}$ are described above, the supporting tetrahedron of
   $\phi_i$, $T_{i_2}$ is any one of two supporting tetrahedra of $\phi_i$,
    $\b E_m=\p{1\\ 0\\ 0}, \ \p{0\\ 1\\ 0}$ or $\p{0\\0\\1}$,  
    $\b n_j$ is a unit normal vector (either direction but
   fixed)  on triangle $F_j$, $n_1$ is the number of all type one $P_2^{\t{nc}}$
   bubbles, and $n_{nc}$ is the number of all $P_2^{\t{nc}}$ bubbles.

The nodes of \cq \ nodal basis functions are vertices and mid-edge points.
Let $\{\b x_i,\ i=1,\dots,4 \}$ and $\{\b e_i, \ i=1,\dots,6\}$ be the
vertices and the mid-edge points of $T\in\mathcal{T}_h$, respectively.
A nodal basis function $\Psi_{i} \in P_2(T)$ satisfies 
\a{ \Psi_i(\b m_j) =\delta_{ij}, }
where the nodes $\{\b m_i, \ i=1,\dots,10\}=\{\b x_1,\dots,\b x_4,  \b e_1,\dots,\b e_6\}$.
The 3 global basis functions at a node $\b m_{i_1}$ (a vertex or a mid-edge point) is defined by
\an{\label{g-1} \psi_i=\begin{cases} \Psi_j \b E_m, & \t{on }\ T, \ \t{if } \ \b m_j=\b m_{i_1}
                               \in T, \\
                         0 , & \t{on }\ T, \ \t{if } \ \b m_i \not\in T, \end{cases} }
$i=1,\dots,n_c$, where $\b m_{i_1}$ is a $P_2^{\t{c}}$ Lagrange node,
   $\b E_m=\p{1\\0\\0}, \p{0\\1\\0}$ or $\p{0\\0\\1} $, and
   $n_c$ is the number of global $P_2^{\t{c}}$ basis functions.

The proposed \ncb \ finite element space is defined by
\an{\label{V-h} \ad{ \b V_h & =\Big\{\sum_{i=1}^{n_{nc}} c_i \phi_i \Big\}
            \oplus \Big( P_2^{\t{c}} \cap H^1_0(\Omega)^3 \Big)  \\
    & =\Big\{ \b v_h=\sum_{i=1}^{n_{nc}} c_i \phi_i + \sum_{j=1}^{n_c}d_j \psi_j \Big\},  } }
where $\phi_i$ and $\psi_j$ are defined in \eqref{g-2} and \eqref{g-1}, respectively,
    $n_{n_c}=3n_{n_t}+n_f$ and $n_c=3n_v+3n_e$.
The $P^{\t{dis}}_1$ finite element space, for approximating the pressure, is defined by
\an{\label{P-h}
    P_h =\{ p_h\in L^2_0(\Omega) : p_h|_T\in P_1(T), T\in \mathcal{T}_h \}. }

The \ncb/$P_1^{\t{dis}}$ finite element problem for the Stationary Stokes equations 
    \eqref{e1-1}--\eqref{e1-3} reads:  
    Find $\b u_h\in \b V_h$ and 
    $p\in P_h$  such that 
\an{ \label{f-e-1} && (\nabla_h \b u_h, \nabla_h \b v_h)- (\t{div}_h \b v_h,p_h) &=(\b f,\b v_h)  
             &&   \forall \b v_h \in \b V_h, && \\
        && (\t{div}_h \b u_h,q_h)  &=0   &&  \forall q_h \in P_h,  
   \label{f-e-2} }  where $\nabla_h$ and $\t{div}_h$ denote the
     piecewise gradient and divergence on mesh $\mathcal T_h$ respectively,
     and $\b V_h$ and $P_h$ are defined in \eqref{V-h} and \eqref{P-h},
    respectively.

\begin{lemma} The \ncb \ bubble space and the \cq \ vector space in \eqref{V-h} are
   linearly independent.  In other words, the following $n_{nc}+n_c$ vectors are
   linearly independent,
\a{ \phi_1,\dots,\phi_{n_{nc}},\psi_1,\dots \psi_{n_c}, }
  where $\phi_i$ and $\psi_j$ are defined in \eqref{g-2} and \eqref{g-1},
    respectively.
\end{lemma}
\begin{proof} We need to prove uniqueness of expansion of 
\a{ \b v_h=\sum_{i=1}^{n_{nc}} c_i \phi_i + \sum_{j=1}^{n_c}d_j \psi_j,  }
i.e., if $\b v_h=0$, then all $c_i=d_j=0$.

Let $T\in {\mathcal{T}}_h$ be at the boundary of $\Omega$ 
    with 4 face-triangle $F_i$, $i=1,\dots,4$, where $F_4\subset \partial\Omega$.
On $T$, the 6 \ncb-basis functions are numbered as $\phi_{i_m}$, $m=1,\dots, 6$, 
\a{  \p{\Phi_0\\0\\0}, \p{0\\ \Phi_0\\0}, \p{0\\0\\ \Phi_0}, 
     \p{n_{11} \Phi_1\\ n_{12} \Phi_1\\ n_{13} \Phi_1}, 
     \p{n_{21} \Phi_2\\ n_{22} \Phi_2\\ n_{23} \Phi_2}, 
     \p{n_{31} \Phi_3\\ n_{32} \Phi_3\\ n_{33} \Phi_3}, }
where $\b n_i=\p{n_{i1}\\n_{i2}\\n_{i3}}$ is the fixed unit normal vector on $F_i$.
The seventh \ncb-function $
     \p{n_{41} \Phi_4\\ n_{42} \Phi_4\\ n_{43} \Phi_4}$ has $0$ coefficient
   in the definition of
   $\b V_h$ \eqref{V-h} as it is at the domain boundary.
By definition \eqref{V-h}, the 18 \cq-basis functions on face $F_4$ have 0 coefficients.
The other 12 \cq-basis functions on $T$  are numbered as $\psi_{j_m}$, $m=1,\dots,12$, 
\a{  \p{\Psi_i\\0\\0}, \p{0\\ \Psi_i\\0}, \p{0\\0\\ \Psi_i}, \ i=1,\dots,4, }
where $\Psi_4$ is the \cq-basis at vertex $\b x_4$, and
    $\Psi_1$,  $\Psi_2$ and $\Psi_3$ are the \cq-basis at mid-edge of $\b x_1\b x_4$,
    $\b x_2\b x_4$ and $\b x_3\b x_4$, respectively.
We know that
\an{\label{basis} \Psi_i=4\lambda_i\lambda_4, \ i=1,\dots,3; \
           \Psi_4 = \lambda_4(1-2\lambda_1-2\lambda_2-2\lambda_3). }
Let $\b v_h=\b v_1- \b v_2$ where
\a{ \b v_1=\sum_{m=1}^{6} c_{ m} \phi_{i_m}, \quad
    \b v_2=\sum_{m=1}^{12}d_{ m} \psi_{j_m}.}
There are $6+12=18$ coefficients/unknowns.  We have $3\times (3\times 4+1)=39$ equations to 
  determine these 18 unknowns.
  For each component of three components, its integral on each face against each 2D $P_1$ polynomial is zero
    which gives 3 equations on each face-triangle, and its integral on the tetrahedron is zero, which gives
    one equation.  Thus we have $3\times (3\times 4+1)=39$ equations.
    
The three components of $\b v_1$ and of $\b v_2$
    are $P_2$ polynomials and their dofs of \eqref{dof} are same.
 On $F_1$, for the first component of $\b v_1$ and $\b v_2$,  by \eqref{basis}, we get
\a{ \frac 1{|F_1|} \int_{F_1} (\b v_2)_1 \lambda_2 dS &
       = \frac{d_2}{15}+ \frac{d_3}{30} +\frac{d_4}{30} = c_4n_{11},  \\
    \frac 1{|F_1|} \int_{F_1} (\b v_2)_1 \lambda_3 dS &
       = \frac{d_2}{30}+ \frac{d_3}{15} +\frac{d_4}{30} = c_4n_{11},  \\
    \frac 1{|F_1|} \int_{F_1} (\b v_2)_1 \lambda_4 dS &
       = \frac{d_2}{30}+ \frac{d_3}{30} +\frac{d_4}{15} = c_4n_{11}, }
where the normal vector on $F_1$, $\b n_1=(n_{11},n_{12},n_{13})$.
We use these three equations to eliminate three unknowns to get
\a{ d_2=d_3=d_4=\frac{15}2 c_4n_{11}. }
Similarly, by computing $\int_{F_2} (\b v_2)_1 \lambda_m dS$ and
    $\int_{F_3} (\b v_2)_1 \lambda_m dS$, we get
\a{ d_1& =d_3=d_4=\frac{15}2 c_5n_{21}, \\
    d_1& =d_2=d_4=\frac{15}2 c_6n_{31}. }
Together, \a{ d_1& =d_2=d_3=d_4=\frac{15}2 c_4n_{11}=\frac{15}2 c_5n_{21}
           =\frac{15}2 c_6n_{31}. }
By computing $\int_{F_1} (\b v_2)_2 \lambda_m dS$, $\int_{F_2} (\b v_2)_2 \lambda_m dS$ and
    $\int_{F_3} (\b v_2)_2 \lambda_m dS$, we get
\a{ d_5=d_6=d_7=d_8=\frac{15}2 c_4n_{12}=\frac{15}2 c_5n_{22}
           =\frac{15}2 c_6n_{32}. }
By computing the moments of the third component,  we get
\a{ d_9=d_{10}=d_{12}=d_{12}=\frac{15}2 c_4n_{13}=\frac{15}2 c_5n_{23}
           =\frac{15}2 c_6n_{33}. }
Because $\b n_1$ and $ \b n_2$ are linearly independent,  it follows that
\a{  \frac{15c_4}2 \b n_1 = \frac{15c_5}2 \b n_2 
   \ \Rightarrow \ \frac{15c_5}2 = 0,}
and   \a{ d_1=\dots=d_{12}=c_4=c_5=c_6=0.  }
Evaluating $\b v_h$ at the bary-center $\b e_T$,  we get
\a{  \b v_h(\b e_T) &=\frac 14\p{d_1 + d_2+d_3\\
                               d_5 + d_6+d_7\\
                               d_9+ d_{10}+ d_{12} } -\frac 18 \p{d_4\\d_8\\d_{12}}
                   -\p{c_1\\c_2\\c_3}=-\p{c_1\\c_2\\c_3}=\b 0. }
It follows that \a{ c_1& =c_2=c_3=0. } 
All coefficients of $\b v_h$ on $T$ are zero.  

On the next $T_1$ which shares a common
  triangle $F_1$ with $T$,  all 19 coefficients of $\b v_h|_{T_1}$
   for the basis functions on $F_1$ are zero as
  they are also the coefficients of $\b v_h|_T$.
Thus, $T_1$ is like $T$ where the boundary basis functions on $F_1$ do not appear in the
   linear expansion of $\b v_h|_{T_1}$.  
Repeating the above work for every tetrahedron sequentially,  
    we can exhaust all $T\in \mathcal{T}_h$ to show
   that the coefficients of $\b v_h|_{T}$ on all $T$ are zero.  
The proof is complete.
\end{proof}

\section{The stability }

We prove the inf-sup condition next.

\begin{lemma} There is a positive constant $C$ independent of $h$, such that
\an{\label{inf} \inf_{p_h \in P_h} \sup_{\b v_h\in \b V_h}
     \frac { (\t{div}_h \b v_h, p_h) }{ \|\b v_h\|_{1,h} \|p_h\|_0 } \ge C,
} where $P_h$ and $\b V_h$ are defined in \eqref{P-h} and \eqref{V-h}, respectively,
   and $\|\b v_h\|_{1,h}^2 = \|\b v_h\|_{0}^2 + \|\nabla_h \b v_h\|_{0}^2$.
   Here $\nabla_h \b v_h$ stands for the piecewise gradient on mesh $\mathcal T_h$.
\end{lemma} 

\begin{proof} Given a $p_h\in P_h\subset L^2_0(\Omega)$, there is a smooth function
   $\b v\in  H^1_0(\Omega )^3$ such that
\a{   \| \b v\|_1 \le C \|p_h\|_0 \quad \t{and}  \quad \t{div} \b v=p_h.  }
Here to avoid technicalities we assume the domain is regular enough that $\b v$ is smooth.
Let $\tilde {\b v}_h$ be the Scott-Zhang \cq\ interpolation of $\b v$ \cite{Scott-Zhang},
\a{ \tilde {\b v}_h = \sum_{i=1}^{n_c}  v_{j_i}(\b n_{i }) \psi_i,} 
where $v_{j_i}$ is the ${j_i}$-th component of $\b v$, i.e. $j_i=1$ or 2 or 3 depending on
   $\psi_i$,  and $\b n_{i }$ is either a vertex or a
  mid-edge point for the basis function $\psi_i$.
Here we mean, by the Scott-Zhang interpolation, that $v_{j_i}$ is a proper local averaging
   value where $\b v$ does not have a nodal value.
The interpolation is stable \cite{Scott-Zhang} that
\a{ \| \tilde {\b v}_h \|_{1,h} \le C  \|  {\b v}  \|_{1}. }
Let a \ncb \ interpolation of $\b v$ be
\a{ \b v_1 = \tilde {\b v}_h + \b v_2, \ \t{where} \ 
    \b v_2=\sum_{j=n_1+1}^{n_{nc}}  s_{j}\phi_j   }
and, noticing that $\int_{F_j} \Phi_j(\lambda_{j_1}+\lambda_{j_2}+\lambda_{j_3})dS=
  \int_{F_j} \Phi_j dS=3|F_j|$ on $F_j$,
 \a{  s_{n_1+j} =\frac{\int_{F_j} (\b v- \tilde {\b v}_h)\cdot \b n_j ds }
                        {3\int_{F_j}  ds }, \ j=1,\dots, n_f. }
Here $\b n_j$ is the fixed unit normal vector on an interior triangle
    $F_j$, used in definition of $\phi_j$.
Then \a{  \| {\b v}_2 \|_{1,h} \le C ( \| \tilde {\b v}_h \|_{1,h}
                   +  \|  {\b v}- \tilde {\b v}_h \|_{1,h})
        \le C \|  {\b v}  \|_{1}. }

On a tetrahedron $T\in\mathcal{T}_h$,  let 
\a{ p_1&=\t{div}(\b v-\b v_1)\in P_1(T), \\
    \int_{T} p_1 d\b x &= \int_{\partial T} (\b v-\b v_1)\cdot \b n_T dS =0, }
where $\b n_T$ is the unit outward normal vector on each face of $T$.
Let $p_1=a_1\lambda_1+a_2\lambda_2+a_3\lambda_3+a_4\lambda_4$. By
\a{  \int_T p_1 d\b x &= \frac {h^3}{24} (a_1+a_2+a_3+a_4) = 0, \\
             p_1 & =a_1\lambda_1+a_2\lambda_2+a_3\lambda_3+(-a_1-a_2-a_3)\lambda_4\\
                 &= a_1(\lambda_1-\lambda_4)+a_2(\lambda_2-\lambda_4)+a_3(\lambda_3-\lambda_4).
  }
On the other side, by the chair rule, we have
\a{ \nabla \Phi_0 &= \p{-\b n_1/h_1 &
                       -\b n_2/h_2 &
                       -\b n_3/h_3}  \nabla_{\lambda}\Phi_0 \\
                  &= M^T \p{ -8(\lambda_1-\lambda_4)\\-8(\lambda_2-\lambda_4)\\
                       -8(\lambda_3-\lambda_4)}, }
where $M$ is the Jacobian matrix for the mapping, 
   $\b n_i$ is the outward unit normal vector to face $F_i$, and
    $h_i$ is the distance from $\b x_i$ to triangle $F_i$.
$\det(M^{-1})=6|T|\ne0$.   On $T$,  let a \ncb-bubble be 
\a{ \b v_3|_T = \sum_{i=1}^3 c_i \phi_i, \ \t{where} \ 
      \p{c_1\\c_2\\c_3}=-\frac 18 M^{-1}\p{a_1\\a_2\\a_3}. }
Together, we construct a
\a{ \b v_h = \b v_1+\b v_3 }
such that
\a{ \t{div}\b v_h|_T &= \t{div}(\tilde{\b v}+\b v_2) +
                 \t{div}\b v_3=p_h-p_1+ \t{div}\b v_3=p_h, 
        \ T\in \mathcal{T}_h, \\
     \|\b v_h\|_{1,h} &\le \|\b v_1\|_{1,h}+ \| \b v_3\|_{1,h} \le C
          \|\b v \|_{1 } + C\|\b v-\b v_1\|_{1,h}\\ &\le   C
          \|\b v \|_{1 } \le C\|p_h\|_0.  }
The lemma is proved.
\end{proof}

\begin{lemma}The linear system of finite element equations \eqref{f-e-1}--\eqref{f-e-2}
   has a unique solution
     $(\b u_h,p_h)\in \b V_h \times P_h$, where $P_h$ and $\b V_h$ are defined in \eqref{P-h} and \eqref{V-h}, respectively.
\end{lemma}  

\begin{proof} For a square system of finite equations,  we only need to prove the uniqueness.
  Let $\b f=\b 0$ in \eqref{f-e-1}.
Letting $\b v_h=\b u_h$ in \eqref{f-e-1} and $q_h=p_h$ in \eqref{f-e-2}, 
      we add the two equations \eqref{f-e-1}    and \eqref{f-e-2}
   to get
\a{  \|\nabla_h \b u_h \|_0 = 0.  } 
$\b u_h$ would be constant vectors on each $T$.  Because $\b u_h\in\ncb$ and it has a
    zero boundary
   condition,  $\b u_h=0$.   
  By the inf-sup condition \eqref{inf} and \eqref{f-e-1}, we have
\a{ \|p_h\|_0 & \le \frac 1C \sup_{\b v_h\in \b V_h} 
             \frac { (\t{div}_h \b v_h, p_h) }{ \|\b v_h\|_{1,h} } \\
              & = \frac 1C \sup_{\b v_h\in \b V_h} 
             \frac { (\nabla_h \b u_h, \nabla_h  \b v_h) }{ \|\b v_h\|_{1,h} } \\
              & = \frac 1C \sup_{\b v_h\in \b V_h} 
             \frac { 0 }{ \|\b v_h\|_{1,h} } =0. } 
The lemma is proved.
\end{proof}
         
\def\bi#1{\langle #1 \rangle_{\mathcal{F}_h}} 
\section{The convergence}

\begin{theorem}
Let $(\b{u},p)\in (H^3(\Omega) \cap H^1_0 (\Omega))^3 \times (H^2 (\Omega)\cap L^2_0 (\Omega))$
    be the solution of the stationary Stokes problem \eqref{e1-1}--\eqref{e1-3}. 
Let $(\b{u}_h,p_h)\in \b V_h \times P_h $
    be the solution of the finite element problem   \eqref{f-e-1}--\eqref{f-e-2}.
It holds that   
\an{ \label{h1} 
   \|\b{u}-\b{u}_{h}\|_{1,h} +\|p-p_{h}\|_{0} & \le 
      C h^2 ( |\b{u} |_{3}+|p|_{2} ). } 
\end{theorem}
\begin{proof} Multiplying \eqref{e1-1} by $\b v_h\in \b V_h$ and
  doing integration by parts,  we get
\an{\label{i1}  (\b f,\b v_h) &=
     (\nabla\b u,\nabla_h \b v_h) -(\t{div}_h\b v_h, p)
     - \bi{\nabla\b u \cdot\b n,[\b v_h]} + \bi{p,[\b v_h]\cdot\b n},
} where $\b n$ is a fixed unit normal vector on face triangle $F$, 
    $\mathcal{F}_h $ is the set of face triangles in the grid $\mathcal{T}_h $ and
\an{\label{b-i} \bi{f,g} = \sum_{F\in\mathcal{F}_h} \langle f,g \rangle_F.  }
Let $I_h \b u$ be the nodal interpolation of the smooth function $\b u$ in the \cq \ space.
Let $\Pi_h p$ be the element-wise $L^2$-projection of $p$ on the space  $P_h$ on
   grid $\mathcal{T}_h$.
Let $\Pi_h^b p$ be the $L^2$-projection of $p$ on the $P^{\t{dis}}_1(\mathcal{F}_h ) $ space
  i.e., $\Pi_h^b p|_F\in P_1(F)$ satisfying
\a{ \langle \Pi_h^b p, q \rangle_F = \langle p, q  \rangle_F \quad
    \forall q \in P_1(F), \ F\in \mathcal{F}_h. }
 Subtracting \eqref{i1} from \eqref{f-e-1}, we get
\an{\label{i1-2} &\quad \ 
     (\nabla_h (\b u-\b u_h), \nabla_h \b v_h )- (\t{div}_h\b v_h, p-p_h) \\
   \nonumber  &=  \bi{\nabla{\b u}\cdot\b n,[\b v_h]} -\bi{p,[\b v_h]\cdot\b n},
} and
\an{\label{i2} 
 &\quad \ (\nabla_h (I_h\b u-\b u_h), \nabla_h \b v_h )- (\t{div}_h\b v_h, \Pi_h p-p_h)\\
    \nonumber &= 
     (\nabla_h (I_h\b u-\b u), \nabla_h \b v_h )- (\t{div}_h\b v_h, \Pi_h p-p)\\
     \nonumber  &\quad \ +\bi{\nabla{\b u}\cdot\b n,[\b v_h]} -\bi{p,[\b v_h]\cdot\b n}.
   } By \eqref{e1-2} and \eqref{f-e-2}, we have
\an{\nonumber (\t{div}_h (\b u -\b u_h), q_h) &= 0, \\
  \label{i3}  (\t{div}_h (I_h\b u -\b u_h), q_h) &=  (\t{div}_h (I_h\b u -\b u ), q_h).
} 
Letting $\b v_h=I_h\b u -\b u_h$ in \eqref{i2} and $q_h=\Pi_h p -p_h$ in \eqref{i3},
  we add the two equations to get
\an{\label{i4} &\quad \ 
 (\nabla_h (I_h\b u-\b u_h), \nabla_h (I_h\b u-\b u_h))\\
    \nonumber &= 
     (\nabla_h (I_h\b u-\b u), \nabla_h (I_h\b u-\b u_h))
     - (\t{div}_h(I_h\b u-\b u_h), \Pi_h p-p)\\
     \nonumber  &\quad \ +\bi{\nabla{\b u}\cdot\b n,[I_h\b u-\b u_h]}
       -\bi{p,[I_h\b u-\b u_h]\cdot\b n}\\
     \nonumber  &\quad \ +(\t{div}_h (I_h\b u -\b u ),\Pi_h p -p_h). } 
We estimate the 5 terms on the right-hand side of \eqref{i4}.
\an{\label{i4-1} |(\nabla_h (I_h\b u-\b u), \nabla_h (I_h\b u-\b u_h))|
          &\le C \|\nabla_h (I_h\b u-\b u)\|_0 \; \| \nabla_h (I_h\b u-\b u_h)\|_0 \\
 \nonumber&\le C h^2 |\b u|_3 \; \|\nabla_h (I_h\b u-\b u_h)\|_0  \\
  \nonumber&\le C h^4 |\b u|_3 ^2  + \|\nabla_h (I_h\b u-\b u_h)\|_0^2 .  }
For the second term on the right-hand side of \eqref{i4}, we have
\an{\label{i4-2} |(\t{div}_h(I_h\b u-\b u_h), \Pi_h p-p)|
           &\le  \| \t{div}_h(I_h\b u-\b u_h)\|_0 \; \| \Pi_h p-p\|_0 \\
  \nonumber&\le Ch^2 |p|_2 \; \|\nabla_h (I_h\b u-\b u_h)\|_0\\
  \nonumber&\le C h^4 |p|_2^2  + \frac 12 \|\nabla_h (I_h\b u-\b u_h)\|_0^2.  }
We bound the third  term on the right-hand side of \eqref{i4} as,  
\an{\label{i4-3} |\bi{\nabla{\b u}\cdot\b n,[I_h\b u-\b u_h]}|
           &=| \bi{2(\nabla{\b u}- \Pi_h^b\nabla{\b u})\cdot\b n,[I_h\b u-\b u_h]}| \\
  \nonumber&\le C   \|\nabla{\b u}-\Pi_h  \nabla{\b u} \|_0\; |I_h\b u-\b u_h|_{1,h} \\
  \nonumber&\le C h^2  | \b u |_3 \; \|\nabla_h (I_h\b u-\b u_h)\|_0 \\
  \nonumber&\le C h^4 |\b u|_3 ^2  + \|\nabla_h (I_h\b u-\b u_h)\|_0^2 . 
    } Here (and later) we used a trace-inequality, cf. \cite{Zhang-Jump}.
Similarly, we estimate the 4th term on the right-hand side of \eqref{i4},
\an{\label{i4-4} | \bi{p,[I_h\b u-\b u_h]\cdot\b n}| 
           &=|\bi{p-\Pi_h^b p,[I_h\b u-\b u_h]\cdot\b n} | \\
  \nonumber&\le C   \|p-\Pi_h p \|_0\; |I_h\b u-\b u_h|_{1,h} \\
  \nonumber&\le C h^2  | p |_2 \; \|\nabla_h (I_h\b u-\b u_h)\|_0 \\
  \nonumber&\le C h^4 |p|_2^2  + \frac 12 \|\nabla_h (I_h\b u-\b u_h)\|_0^2 . 
    }
For the last term on the right-hand side of \eqref{i4}, we have
\an{\label{i4-5} |(\t{div}_h (I_h\b u -\b u ),\Pi_h p -p_h)|
            &\le  \| \t{div}_h(I_h\b u-\b u )\|_0   \; \| \Pi_h p-p\|_0 \\
  \nonumber&\le Ch^4 |\b u|_3 |p|_2.  }
Substituting the last five bounds \eqref{i4-1}--\eqref{i4-5} in \eqref{i4}, it follows that
\an{\label{i5} \|\nabla_h (I_h\b u-\b u_h)\|_0^2\le C h^4 |\b u|_3 ^2 +  C h^4 |p|_2^2  
     + Ch^4 |\b u|_3 |p|_2. }
By  \eqref{i5}, we have
\a{ |\b u-\b u_h|_{1,h} & \le |I_h \b u-\b u_h|_{1,h} + |I_h\b  u-\b u |_{1,h} \\
     & \le C \|\nabla_h (I_h\b u-\b u_h)\|_0 + Ch^2 |\b u|_3 \\
     & \le Ch^2 |p|_2  + Ch^2 |\b u|_3. } 
\def\gdu{\nabla_h (I_h\b u-\b u_h)}\def\gdv{\nabla_h  \b v_h}

Next we estimate the pressure error in \eqref{h1}.
 By \eqref{i5}, \eqref{i4-1},  
     \eqref{i4-5} and \eqref{i4-3}, we get from \eqref{i2} that
\a{ |(\t{div}_h\b v_h, \Pi_h p-p_h)|
 &= | (\gdu, \gdv )\\
     &\quad \ - 
      (\nabla_h (I_h\b u-\b u), \gdv)
        + (\t{div}_h\b v_h, \Pi_h p-p)\\
     &\quad \ - \bi{\nabla{\b u}\cdot\b n,[\b v_h]}
         +\bi{p,[\b v_h]\cdot\b n}|\\
   &\le  Ch^2(|\b u|_3+|p|_2) \|\gdv\|_0  \\
    &\quad \ +Ch^2 |\b u|_{3}\; \|\gdv\|_0 + C h^2|p|_2\; \|\gdv\|_0  \\
    &\quad \  + Ch^2 |\b u|_{3}\; \|\gdv\|_0  + C h^2|p|_2\; \|\gdv\|_0.
   } 
By the inf-sup condition \eqref{inf}, we obtain
\a{ \|\Pi_h p-p_h \|_0 &\le C \inf_{\b v_h\in \b V_h}
   \frac{ (\t{div}_h\b v_h, \Pi_h p-p_h) } { \|\b v_h \|_{1,h} }\\
   &\le C \inf_{\b v_h\in \b V_h}
   \frac{ Ch^2(|\b u|_3+|p|_2)\|\gdv\|_0 } { C \|\gdv\|_0 }\\
   &= Ch^2(|\b u|_3+|p|_2). }
\eqref{h1} follows the triangle inequality.
\end{proof}

\begin{theorem}
Let $(\b{u},p)\in (H^3(\Omega) \cap H^1_0 (\Omega))^d \times (H^2 (\Omega)\cap L^2_0 (\Omega))$
    be the solution of the stationary Stokes problem \eqref{e1-1}--\eqref{e1-3}. 
Let $(\b{u}_h,p_h)\in \b V_h \times P_h $
    be the solution of the finite element problem   \eqref{f-e-1}--\eqref{f-e-2}.
It holds that   
\an{ \label{l2}
   \|\b u -\b u_{h}\|_{0}  & \le  C h^3 \{ |\b{u} |_{3}+|p|_{2}\} . } 
\end{theorem}

 \def\iuh{\b u-\b u_h}

\begin{proof}
We will use the duality argument.  Let 
  $\b w \in H^1_0(\Omega)^d $ and $r \in L^2_0(\Omega)$ such that
  \an{\label{e-w} &&  - \t{div}(\nabla \b w)    + \nabla r
            &= \iuh && \hbox{in } \Omega, && \\
     \label{e-w2}   && \t{div} \b w  &= 0 && \hbox{in } \Omega.  } 
We assume the following regularity,
\an{\label{regularity} |\b w|_2 + |r|_1 \le C \|\iuh\|_0.  }
Multiplying \eqref{e-w} by $(\iuh)$ and
  doing integration by parts,  we get
\an{\label{j1} &\quad \ (\iuh,\iuh) \\
  \nonumber &=
     (\nabla \b w , \nabla_h (\iuh))-(\t{div}_h(\iuh), r)\\
  \nonumber &\quad \
     - \bi{\nabla{\b w}\cdot\b n,[\iuh]} + \bi{r,[\iuh]\cdot\b n} \\
  \nonumber &=
     (\nabla\b w , \nabla_h (\iuh))-(\t{div}_h(\iuh), r)\\
  \nonumber &\quad \
     - \bi{\nabla{\b w}\cdot\b n,[I_h\iuh]} + \bi{r,[I_h\iuh]\cdot\b n},
} where $\bi{\cdot,\cdot} $ is defined in \eqref{b-i}.
Let $\b w_h\in \b V_h$ and  $r_h\in P_h$ be the finite element solution for problem
    \eqref{e-w}--\eqref{e-w2}, satisfying 
\an{\label{f-w} &&
    (\nabla_h   \b w_h , \nabla_h \b v_h )
  - (\t{div}_h \b v_h,r_h) &=(\iuh,\b v_h)  
             &   \forall \b v_h & \in \b V_h, && \\
        && (\t{div}_h \b w_h, q_h)  &=0   &   \forall q_h & \in P_h. 
   \label{f-w2} }  
Letting $\b v_h=\b w_h$ in \eqref{i1-2}, we get
 \an{\label{j1-2}  
    &\quad \ (\nabla_h (\b u-\b u_h), \nabla_h \b w_h)  \\
  \nonumber   &= (\t{div}\b w_h, \Pi_h p-p_h) +
     \bi{\nabla{\b u}\cdot\b n,[\b w_h]} -\bi{p,[\b w_h]\cdot\b n} \\
  \nonumber   &=  
     \bi{\nabla{\b u}\cdot\b n,[\b w_h-I_h\b w_h]} -\bi{p,[\b w_h-I_h\b w_h]\cdot\b n},
} where we use the fact $\t{div}_h\b w_h|_T \in P_1(T)$ and \eqref{f-w2}.

Subtracting \eqref{j1-2} from \eqref{j1},   we get
\an{\label{j3} \|\iuh\|_0^2  &=
     (\nabla_h (\b w-\b w_h), \nabla_h (\iuh))\\ 
  \nonumber &\quad \
     - \bi{\nabla{\b w}\cdot\b n,[I_h\iuh]} + \bi{r,[I_h\iuh]\cdot\b n}\\
  \nonumber   &   \quad \ +
     \bi{\nabla{\b u}\cdot\b n,[\b w_h-I_h\b w_h]} -\bi{p,[\b w_h-I_h\b w_h]\cdot\b n},
   } where we use the fact $(\t{div}_h(\iuh), r )=(\t{div}_h \b u_h , r)
     =(\t{div}_h \b u_h, r-\Pi_h r)=0$, because $\t{div}_h \b u_h|_T\in P_1(T)$.
The five terms on the right-hand side of \eqref{j3} are estimated similarly as those
   five terms on the right-hand side of \eqref{i4}. 
To be specific, they are estimated by the methods in \eqref{i4-1}, \eqref{i4-3}, \eqref{i4-4}, 
   \eqref{i4-3} and \eqref{i4-4}, respectively.
We have, by the quasi-optimal error bound \eqref{h1} and
    the assumed elliptic regularity \eqref{regularity}, from \eqref{j3}, 
\an{\label{j4} \|\iuh\|_0^2  &\le  \|\nabla_h (\b w-\b w_h)\|_0\; 
               \|\nabla_h (\iuh)\|_0 +C h |\b w|_2\;h^2|\b u|_3  \\ 
  \nonumber &\quad \ + Ch|r|_1\;h^2|\b u|_3  +C h^2 |\b u|_3\;h|\b w|_2 + Ch^2|p|_2\;h|\b w|_2\\
   \nonumber &\le C h^3 ( |\b u|_3 + |p|_2) ( |\b w|_2 +  |r|_1) \\
     \nonumber &\le C h^3 ( |\b u|_3 + |p|_2) \|\iuh\|_0. 
   }
\eqref{l2} follows \eqref{j4}. 
\end{proof}

\section{Numerical tests}\lab{s-numerical}

We solve the following 3D Stokes problem on a unit-cube domain $\Omega=(0,1)^3$: 
    Find $\b u \in \b V=H^1_0(\Omega)^3$ and  $p \in P=L_0^2(\Omega)$  such that  
\an{ \ad{   (\nabla\b u, \nabla\b v )- (\d \b v ,p )
     &=\b f 
             && \forall \b v  \in \b V , \\
        (\d \b u, q)  &=0 && \forall q  \in P,} \label{s4} }
where $\b f$ is chosen so that the exact solution is
\an{ \ad{   \b u&=\p{g_{y}-g_z\\ -g_{x}\\ g_x }  \quad \t{and} \quad
                   p=\frac 19 g_{xy}, \quad \t{where} \\
           g&=2^{12} x^2(1-x)^2y^2(1-y)^2z^2(1-z)^2. }
   \label{s3} }

\begin{figure}[ht]\begin{center}\begin{picture}(300,120)(10,0)
\put(0, -50){\includegraphics[width=1.6in]{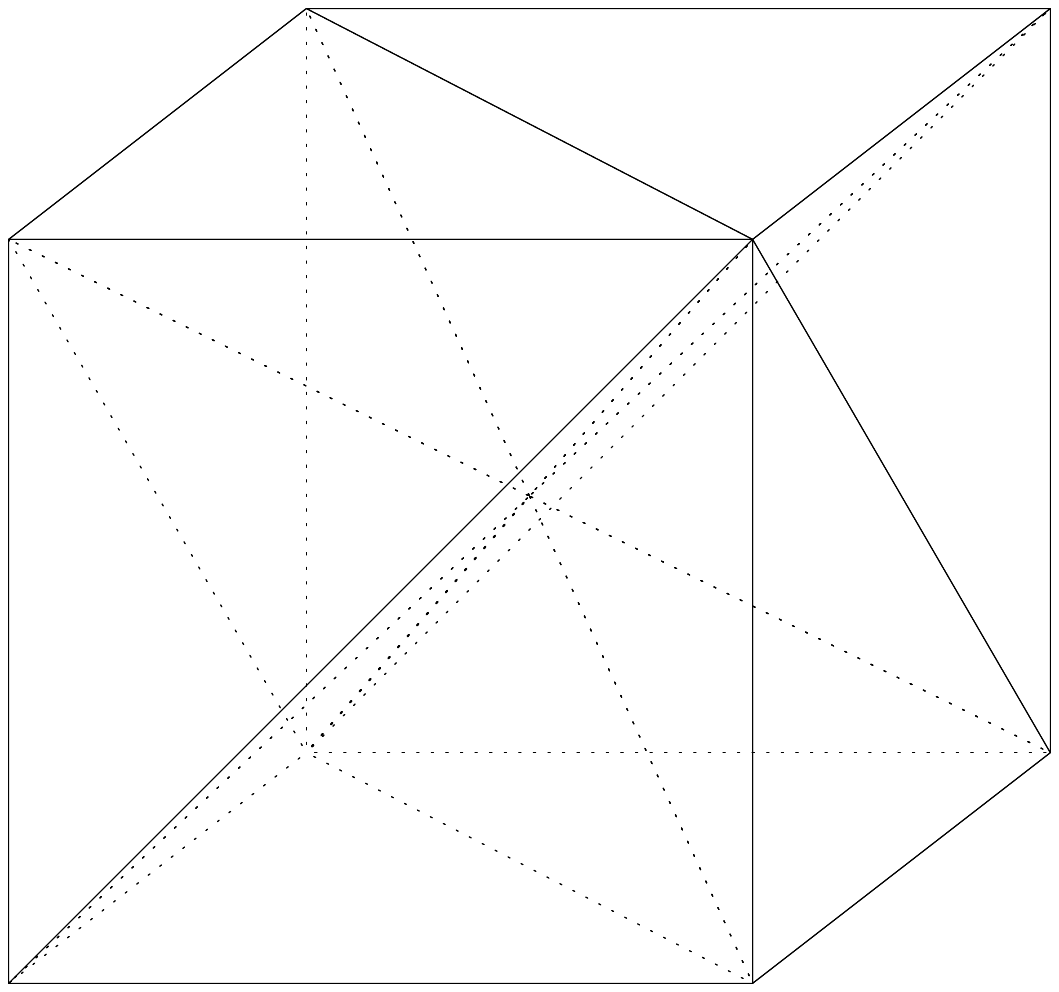}}
\put(100,-50){\includegraphics[width=1.6in]{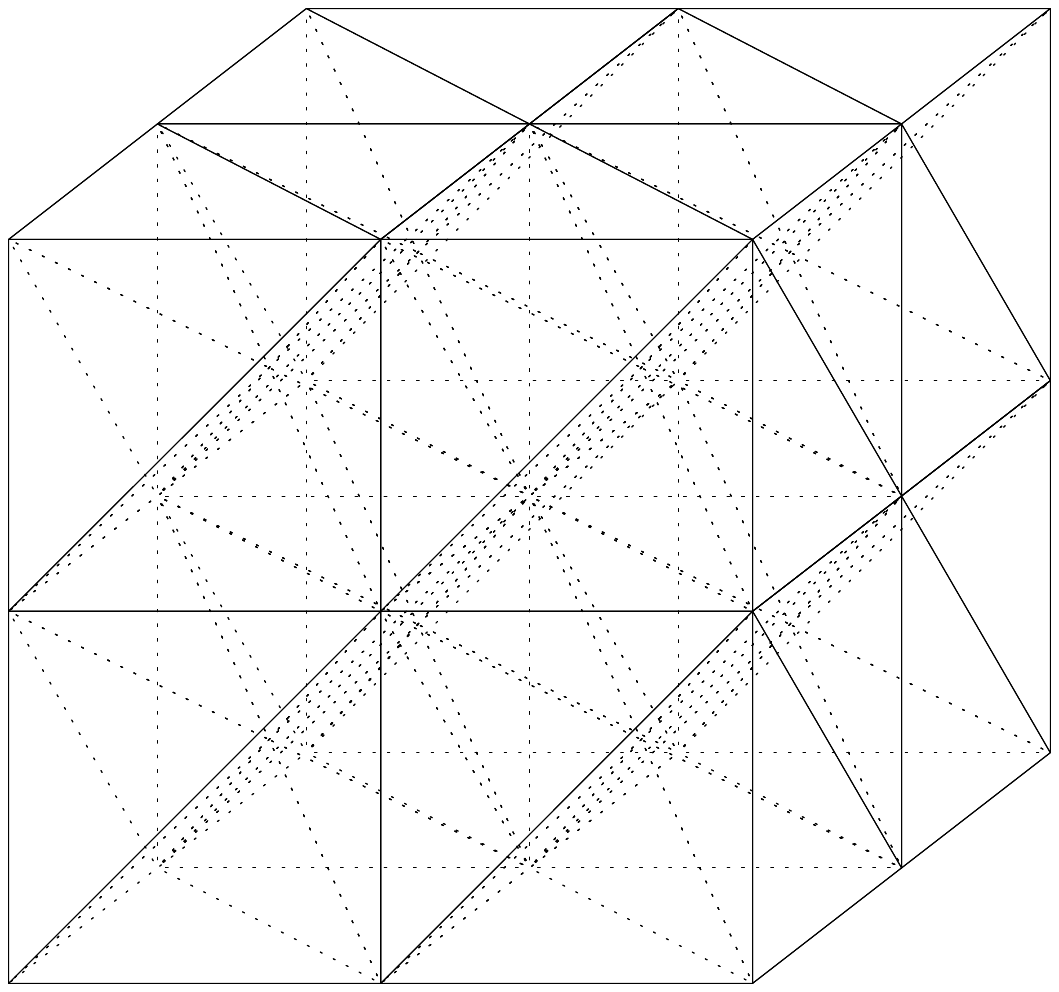}} 
\put(200,-50){\includegraphics[width=1.6in]{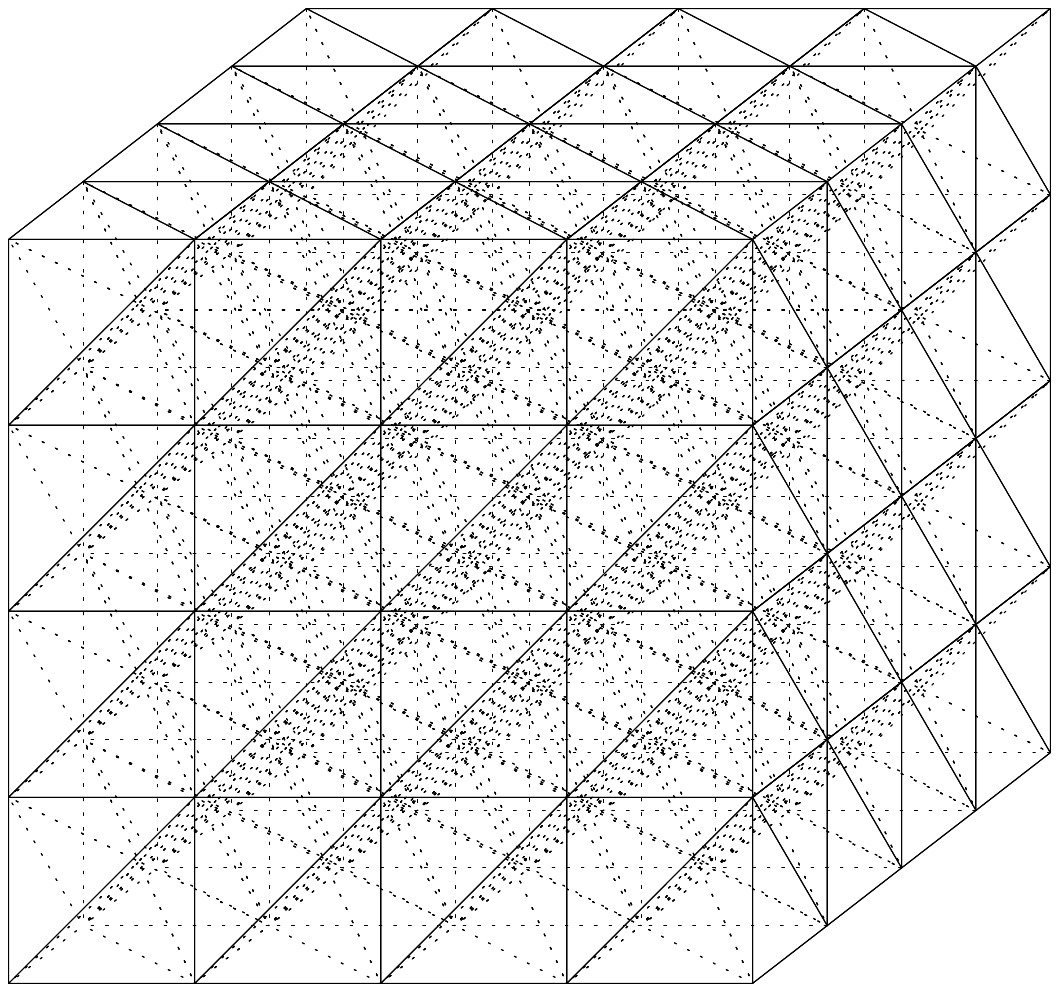}} 
\end{picture}
\caption{ The first three tetrahedral grids for the computation in Tables \ref{t1}.  }
\label{grid}
\end{center}
\end{figure}
  
The 3D tetrahedral grids employed in the computation are illustrated in Figure \ref{grid}.
In Table \ref{t1} we list the errors and the orders of convergence, 
	for the \ncb/$P_1^{\t{dis}}$ finite element \eqref{V-h}--\eqref{P-h},
    in solving problem \eqref{s3}.
We have optimal orders of convergence for all solutions in all norms, 
   confirming the main theorems.  

  \begin{table}[htb]
  \caption{ Error profile by the \ncb/$P^{\t{dis}}_1$ finite element for problem \eqref{s3} on 
     Figure \ref{grid} grids.} \label{t1}
\begin{center}  
   \begin{tabular}{c|rr|rr|rr}  
 \hline 
grid &  $ \|\b u - \b u_h \|_{0}$ & rate &  $ \|\nabla_h (\b u - \b u_h) \|_{0}$ & rate &
   $ \| p - p_h \|_{0}$   & rate   \\ \hline 
 1&   0.495E+00&0.0&   0.563E+01&0.0&   0.363E+01&0.0 \\
 2&   0.409E+00&0.3&   0.568E+01&0.0&   0.535E+01&0.0 \\
 3&   0.575E-01&2.8&   0.197E+01&1.5&   0.190E+01&1.5 \\
 4&   0.741E-02&3.0&   0.545E+00&1.9&   0.548E+00&1.8 \\
 5&   0.919E-03&3.0&   0.140E+00&2.0&   0.144E+00&1.9 \\
 6&   0.114E-03&3.0&   0.352E-01&2.0&   0.365E-01&2.0 \\
\hline 
\end{tabular} \end{center}  \end{table}

\section{Ethical Statement}

\subsection{Compliance with Ethical Standards} { \ }

   The submitted work is original and is not published elsewhere in any form or language.

\subsection{Funding } { \ }

  This research is not supported by any funding agency.

\subsection{Conflict of Interest} { \ }

  There is no potential conflict of interest .

\subsection{Ethical approval} { \ }

  This article does not contain any studies involving animals.
This article does not contain any studies involving human participants.
  
\subsection{Informed consent}  { \ }

This research does not have any human participant.  

\subsection{Availability of supporting data } { \ }

This research does not use any external or author-collected data.

\subsection{Authors' contributions } { \ }

The single author made all contribution.
  
\subsection{Acknowledgments } { \ }

  None.

\end{document}